\documentclass{amsart}

\usepackage{amsmath}
\usepackage{amsfonts}
\usepackage{amssymb,enumerate}
\usepackage{amsthm}
\usepackage[all]{xy}
\usepackage{hyperref}

\newcommand{\Kr}{\operatorname{KR}}
\newcommand{\Na}{\operatorname{Na}}
\newcommand{\NA}{\operatorname{NA}}

\newcommand{\QSpec}{\operatorname{QSpec}}

\newcommand{\Max}{\operatorname{Max}}
\newcommand{\QMax}{\operatorname{QMax}}

\newcommand{\A}{\mathcal{A}}

\newtheorem{thm}{Theorem}[section]
\newtheorem{cor}[thm]{Corollary}
\newtheorem{lem}[thm]{Lemma}
\newtheorem{prop}[thm]{Proposition}

\newtheorem{exam}[thm]{Example}
\newtheorem{rem}[thm]{Remark}

\begin{document}

\bibliographystyle{amsplain}

\date{}

\author{Parviz Sahandi}

\address{Department of Pure Mathematics, Faculty of Mathematical Sciences, University of Tabriz, Tabriz,
Iran and School of Mathematics, Institute for Research in
Fundamental Sciences (IPM), P.O. Box: 19395-5746, Tehran, Iran}
\email{sahandi@ipm.ir}



\thanks{Parviz Sahandi was in part supported by a grant from IPM (No.
91130030)}

\title[Pr\"{u}fer $\star$-multiplication domains]{Characterizations of graded Pr\"{u}fer $\star$-multiplication domains, II}

\begin{abstract} Let $R=\bigoplus_{\alpha\in\Gamma}R_{\alpha}$ be a graded
integral domain and $\star$ be a semistar operation on $R$. For
$a\in R$, denote by $C(a)$ the ideal of $R$ generated by homogeneous
components of $a$ and for$f=f_0+f_1X+\cdots+f_nX^n\in R[X]$, let
$\A_f:=\sum_{i=0}^nC(f_i)$. Let $N(\star):=\{f\in R[X]\mid
f\neq0\text{ and }\A_f^{\star}=R^{\star}\}$. In this paper we study
relationships between ideal theoretic properties of
$\NA(R,\star):=R[X]_{N(\star)}$ and the homogeneous ideal theoretic
properties of $R$. For example we show that $R$ is a graded
Pr\"{u}fer-$\star$-multiplication domain if and only if
$\NA(D,\star)$ is a Pr\"{u}fer domain if and only if  $\NA(R,\star)$
is a B\'{e}zout domain. We also determine when $\NA(R,v)$ is a PID.
\end{abstract}

\maketitle

\section{Introduction}

\subsection{Graded integral domains}

Let $R=\bigoplus_{\alpha\in\Gamma}R_{\alpha}$ be a graded
(commutative) integral domain graded by an arbitrary grading
torsionless monoid $\Gamma$, that is $\Gamma$ is a commutative
cancellative monoid (written additively). Let
$\langle\Gamma\rangle:=\{a-b\mid a,b\in\Gamma\},$ be the quotient
group of $\Gamma$, which is a torsionfree abelian group.

Let $H$ be the saturated multiplicative set of nonzero homogeneous
elements of $R$. Then
$R_H=\bigoplus_{\alpha\in\langle\Gamma\rangle}(R_H)_{\alpha}$,
called the \emph{homogeneous quotient field of $R$}, is a graded
integral domain whose nonzero homogeneous elements are units. An integral ideal $I$ of $R$ is said
to be homogeneous if $I=\bigoplus_{\alpha\in\Gamma}(I\cap
R_{\alpha})$. A fractional ideal $I$ of
$R$ is \emph{homogeneous} if $sI$ is an integral homogeneous ideal
of $R$ for some $s\in H$ (thus $I\subseteq R_H$). An overring $T$ of $R$, with
$R\subseteq T\subseteq R_H$ will be called a \emph{homogeneous
overring} if $T=\bigoplus_{\alpha\in\langle\Gamma\rangle}(T\cap
(R_H)_{\alpha})$. Thus $T$ is a graded integral domain with
$T_{\alpha}=T\cap (R_H)_{\alpha}$. For more on graded integral domains and their divisibility
properties, see \cite{AA2}, \cite{North}.

\subsection{Motivations and results}

Let $D$ be a domain with quotient field $L$, and let $X$ be an
indeterminate over $L$. For each $f\in L[X]$, we let $c(f)$ denote
the content of the polynomial $f$, i.e., the (fractional) ideal of
$D$ generated by the coefficients of $f$. Let $\star$ be a semistar
operation on $D$. If $N_{\star}:=\{g\in D[X]\mid g\neq0\text{ and
}c(g)^{\star}=D^{\star}\}$, then $N_{\star}=
D[X]\backslash\bigcup\{P[X]\mid P\in\QMax^{\star_f}(D)\}$ is a
saturated multiplicative subset of $D[X]$. The ring of fractions
$$\Na(D,\star):=D[X]_{N_{\star}}$$ is called the $\star$-{\it Nagata domain (of $D$ with respect to the
semistar operation} $\star$). When $\star=d$, the identity
(semi)star operation on $D$, then $\Na(D,d)$ coincides with the
classical Nagata domain $D(X)$ (as in, for instance \cite[page
18]{Na}, \cite[Section 33]{G} and \cite{FL}).

The $\star$-Nagata ring has an important overring which is the \emph{Kronecker function ring with respect to $\star$} defined as follows (see, \cite[Section 32]{G}, \cite{FL}, \cite{FL2} and \cite{FL3}).

$$
\mbox{Kr}(D,\star):=\left\{\frac{f}{g}\bigg|  \begin{array}{l}
f,g\in D[X], g\neq0,\text{ and there is }0\neq h\in D[X]
\\\text{ such that } c(f)c(h)\subseteq(c(g)c(h))^{\star} \end{array} \right\}.
$$

The $\star$-Nagata ring and Kronecker function ring have many
interesting ring-theoretic properties. For example $D$ is a
P$\star$MD if and only if $\Na(D,\star)$ is a Pr\"{u}fer domain if
and only if $\Na(D,\star)$ is a B\'{e}zout domain if and only if
$\Na(D,\star)=\mbox{Kr}(D,\star)$. Also $\mbox{Kr}(D,\star)$ is
always a B\'{e}zout domain and that every invertible ideal of
$\Na(D,\star)$ is principal.

Let $R=\bigoplus_{\alpha\in\Gamma}R_{\alpha}$ be a graded integral
domain and $\star$ be a semistar operation on $R$. The aim of this
paper is to introduced a $\star$-Nagata and a Kronecker function
ring with respect to $\star$, on $R$ which is compatible with the
graded structure of the base ring $R$. More precisely in Section 2,
we introduced a $\star$-Nagata ring, denoted by $\NA(R,\star)$ and
we study its properties. For example  we show that
$\Max(\NA(R,\star))=\{Q\NA(R,\star)\mid Q\in
h\text{-}\QMax^{\star_f}(R)\}$ and that each invertible ideal of
$\NA(R,\star)$ is principal, and
$I^{\widetilde{\star}}=I\NA(R,\star)\cap R_H$ for every nonzero
finitely generated homogeneous ideal $I$ of $R$. We also introduced
a Kronecker function ring denoted by $\Kr(R,\star)$, and show that
$\Kr(R,\star)$ is a B\'{e}zout domain and that
$I^{\star_a}=I\Kr(R,\star)\cap R_H$ for every nonzero finitely
generated homogeneous ideal $I$ of $R$.

In Section 3, we are able to give the related characterizations of
graded Pr\"{u}fer-$\star$-multiplication domains. More precisely
among other things, we show that $R$ is a graded
Pr\"{u}fer-$\star$-multiplication domain if and only if
$\NA(R,\star)$ is a Pr\"{u}fer domain if and only if $\NA(R,\star)$
is a B\'{e}zout domain if and only if $\NA(R,\star)=\Kr(R,\star)$ if
and only if every (principal) ideal of $\NA(R,\star)$ is extended
from a homogeneous ideal of $R$. Also we show that $R$ is a graded
Krull domain if and only if $\NA(R,v)$ is a PID.

\subsection{Definitions related to semistar operations}

To facilitate the reading of the paper, we review some basic facts
on semistar operations. Let $D$ be an integral domain with quotient
field $L$. Let $\overline{\mathcal{F}}(D)$ denote the set of all
nonzero $D$-submodules of $L$, $\mathcal{F}(D)$ be the set of all
nonzero fractional ideals of $D$, and $f(D)$ be the set of all
nonzero finitely generated fractional ideals of $D$. Obviously,
$f(D)\subseteq\mathcal{F}(D)\subseteq\overline{\mathcal{F}}(D)$. As
in \cite{OM}, a {\it semistar operation on} $D$ is a map
$\star:\overline{\mathcal{F}}(D)\rightarrow\overline{\mathcal{F}}(D)$,
$E\mapsto E^{\star}$, such that, for all $0\neq x\in L$, and for all
$E, F\in\overline{\mathcal{F}}(D)$, the following three properties
hold: ($\star_1$)  $(xE)^{\star}=xE^{\star}$; ($\star_2$):
$E\subseteq F$ implies that $E^{\star}\subseteq F^{\star}$;
($\star_3$)  $E\subseteq E^{\star}$; and ($\star_4$)
$E^{\star\star}:=(E^{\star})^{\star}=E^{\star}$.

A semistar operation $\star$ is called a \emph{(semi)star operation on $D$}, if $D^{\star}=D$. Let $\star$ be a semistar operation on the domain $D$. For every
$E\in\overline{\mathcal{F}}(D)$, put $E^{\star_f}:=\cup F^{\star}$,
where the union is taken over all finitely generated $F\in f(D)$
with $F\subseteq E$. It is easy to see that $\star_f$ is a semistar
operation on $D$, and ${\star_f}$ is called \emph{the semistar
operation of finite type associated to} $\star$. Note that
$(\star_f)_f=\star_f$. A semistar operation $\star$ is said to be of
\emph{finite type} if $\star=\star_f$; in particular ${\star_f}$ is
of finite type. We say that a nonzero ideal $I$ of $D$ is a
\emph{quasi-$\star$-ideal} of $D$, if $I^{\star}\cap D=I$; a
\emph{quasi-$\star$-prime} (ideal of $D$), if $I$ is a prime
quasi-$\star$-ideal of $D$; and a \emph{quasi-$\star$-maximal}
(ideal of $D$), if $I$ is maximal in the set of all proper
quasi-$\star$-ideals of $D$. Each quasi-$\star$-maximal ideal is a
prime ideal. It was shown in \cite[Lemma 4.20]{FH} that if
$D^{\star} \neq L$, then each proper quasi-$\star_f$-ideal of $D$ is
contained in a quasi-$\star_f$-maximal ideal of $D$. We denote by
$\QMax^{\star}(D)$ (resp., $\QSpec^{\star}(D)$) the set of all
quasi-$\star$-maximal ideals (resp., quasi-$\star$-prime ideals) of
$D$.

If $\star_1$ and $\star_2$ are semistar operations on $D$, one says
that $\star_1\leq\star_2$ if $E^{\star_1}\subseteq E^{\star_2}$ for
each $E\in\overline{\mathcal{F}}(D)$ (cf. \cite[page 6]{OM}). This
is equivalent to saying that
$(E^{\star_1})^{\star_2}=E^{\star_2}=(E^{\star_2})^{\star_1}$ for
each $E\in\overline{\mathcal{F}}(D)$ (cf. \cite[Lemma 16]{OM}).

Given a semistar operation $\star$ on $D$, it is possible to construct a
semistar operation $\widetilde{\star}$, which is stable and of
finite type defined as follows: for each
$E\in\overline{\mathcal{F}}(D)$,
$$
E^{\widetilde{\star}}:=\{x\in K\mid xJ\subseteq E,\text{ for some
}J\subseteq R, J\in f(D), J^{\star}=D^{\star}\}.
$$

The most widely studied (semi)star operations on $D$ have been the
identity $d$, $v$, $t:=v_f$, and $w:=\widetilde{v}$ operations,
where $A^{v}:=(A^{-1})^{-1}$, with $A^{-1}:=(D:A):=\{x\in
K|xA\subseteq D\}$. If $\star$ is a (semi)star operation on $D$,
then $d\leq\star\leq v$.

Let $\star$ be a semistar operation on an integral domain $D$. We
say that $\star$ is an \emph{\texttt{e.a.b.} (endlich arithmetisch
brauchbar) semistar operation} of $D$ if, for all $E, F, G\in f(D)$,
$(EF)^{\star}\subseteq(EG)^{\star}$ implies that $F^{\star}\subseteq
G^{\star}$ (\cite[Definition 2.3 and Lemma 2.7]{FL2}). We can
associate  to any semistar operation $\star$ on $D$, an
\texttt{e.a.b.} semistar operation of finite type $\star_a$ on $D$,
called the \emph{\texttt{e.a.b.} semistar operation associated to
$\star$}, defined as follows for each $F\in f(D)$ and for each $E\in
\overline{F}(D)$:
\begin{align*}
F^{\star_a}:=&\bigcup\{((FH)^{\star}:H^{\star})\mid  H\in f(D)\},\\[1ex]
E^{\star_a}:=&\bigcup\{F^{\star_a}\mid  F\subseteq E, F\in f(D)\}
\end{align*}
\cite[Definition 4.4 and Proposition 4.5]{FL2}. It is known that
$\star_f\leq\star_a$ \cite[Proposition 4.5(3)]{FL2}. Obviously
$(\star_f)_a=\star_a$. Moreover, when $\star=\star_f$, then $\star$
is \texttt{e.a.b.} if and only if $\star=\star_a$ \cite[Proposition
4.5(5)]{FL2}.

Let $\star$ be a semistar operation on a domain $D$. Recall from
\cite{FJS} that, $D$ is called a \emph{Pr\"{u}fer
$\star$-multiplication domain} (for short, a P$\star$MD) if each
finitely generated ideal of $D$ is \emph{$\star_f$-invertible};
i.e., if $(II^{-1})^{\star_f}=D^{\star}$ for all $I\in f(D)$. When
$\star=v$, we recover the classical notion of P$v$MD; when
$\star=d_D$, the identity (semi)star operation, we recover the
notion of Pr\"{u}fer domain.

Let $R=\bigoplus_{\alpha\in\Gamma}R_{\alpha}$ be a graded integral
domain with quotient field $K$, and $\star$ be a semistar operation
on $R$. We say that $\star$ is \emph{homogeneous preserving} if
$\star$ sends homogeneous fractional ideals to homogeneous ones. It
is known that $\widetilde{\star}$ is homogeneous preserving
\cite[Proposition 2.3]{S}, and that if $\star$ is homogeneous
preserving, then so is $\star_f$ \cite[Lemma 2.4]{S}. Denote by
$h$-$\QSpec^{\star}(R)$ the homogeneous elements of
$\QSpec^{\star}(R)$ and let $h$-$\QMax^{\star}(R)$ denotes the set
of ideals of $R$ which are maximal in the set of all proper
homogeneous quasi-$\star$-ideals of $R$. It is shown that if
$R^{\star}\subsetneq R_H$ and $\star=\star_f$ homogeneous
preserving, then $h$-$\QMax^{\star_f}(R)(\subseteq
h$-$\QSpec^{\star}(R))$ is nonempty and each proper homogeneous
quasi-$\star_f$-ideal is contained in a homogeneous maximal
quasi-$\star_f$-ideal \cite[Lemma 2.1]{S}, and that
$h$-$\QMax^{\star_f}(R)=h$-$\QMax^{\widetilde{\star}}(R)$
\cite[Proposition 2.5]{S}.

Recall from \cite{S} that $R$ is called a \emph{graded Pr\"{u}fer
$\star$-multiplication domain (GP$\star$MD)} if every nonzero
finitely generated homogeneous ideal of $R$ is a
$\star_f$-invertible. When $\star=v$ we have the classical notion of
a \emph{graded Pr\"{u}fer $v$-multiplication domain (GP$v$MD)}
\cite{AA}. Also when $\star=d$, a GP$d$MD is called a
\emph{graded-Pr\"{u}fer domain} \cite{AC}. Although $R$ is a GP$v$MD
if and only if $R$ is a P$v$MD \cite[Theorem 6.4]{AA}, Anderson and
Chang \cite[Example 3.6]{AC} provided an example of a
graded-Pr\"{u}fer domain which is not Pr\"{u}fer. It is known that
$R$ is a GP$\star$MD if and only if $R_{H\backslash P}$ is a
graded-Pr\"{u}fer domain for all $P\in
h$-$\QMax^{\widetilde{\star}}(R)$ if and only if $R_P$ is a
valuation domain for all $P\in h$-$\QMax^{\widetilde{\star}}(R)$
\cite[Theorem 4.4]{S}, and that the notions of GP$\star$MD,
GP$\star_f$MD and GP$\widetilde{\star}$MD coincide.

\section{Nagata and Kronecker function rings}

\subsection{Nagata ring}

Let $R=\bigoplus_{\alpha\in\Gamma}R_{\alpha}$ be a graded integral domain with quotient field $K$, and let $H$ be the saturated multiplicative set of nonzero homogeneous elements of $R$.

If $a\in R$, we denote by $C(a)$ the homogeneous ideal of $R$
generated by homogeneous components of $a$. For
$f=f_0+f_1X+\cdots+f_nX^n\in R[X]$, we define the \emph{homogeneous
content ideal} of $f$ by $\A_f:=\sum_{i=0}^nC(f_i)$. It is easy to
see that $\A_f$ is a homogeneous finitely generated ideal of $R$,
and that if $R$ has trivial grading, then $\A_f=c(f)$ (note that in
\cite[Section 28]{G}, $c(f)$ is denoted by $A_f$). It is easy to see
that $\A_{f+g}\subseteq\A_f+\A_g$ and $a\A_f=\A_{af}$ for $f,g\in
R[X]$ and a homogeneous element $a$. Also it can be seen that
$\A_{fg}\subseteq\A_f\A_g$ for $f,g\in R[X]$. Indeed assume that
$f=\sum_{i=0}^{r}a_iX^i$ and $g=\sum_{j=0}^{s}b_jX^j$. Then
$fg=\sum_{\ell=0}^{r+s} c_{\ell}X^{\ell}$, where
$c_{\ell}:=\sum_{i+j=\ell}a_ib_j$. Since $C(a_ib_j)\subseteq
C(a_i)C(b_j)$, we have $\A_{fg}=\sum_{\ell=0}^{r+s}
C(c_{\ell})\subseteq\Sigma_{\ell=0}^{r+s}\Sigma_{i+j=\ell}
C(a_ib_j)\subseteq\sum_{\ell=0}^{r+s}\sum_{i+j=\ell}
C(a_i)C(b_j)=(\sum_{i=0}^{r} C(a_i))(\sum_{j=0}^{s}
C(b_j))=\A_f\A_g$.

We begin with the following graded analogue of Dedekind-Mertens lemma.

\begin{prop}\label{dm} Assume that $0\neq f,g\in R_H[X]$.
Then there is an integer $m\geq1$ such that
$\A_f^{m+1}\A_g=\A_f^m\A_{fg}$. In particular if $\star$ is an
\texttt{e.a.b.} semistar operation on $R$, then
$\A_{fg}^{\star}=(\A_f\A_g)^{\star}$.
\end{prop}

\begin{proof} Let $f=\sum_{i=0}^{n}a_iX^i$ and
$g=\sum_{j=0}^{k}b_jX^j$. Assume that $a_i=a_{i1}+\cdots+a_{in_i}$,
where $a_{ij}\in (R_H)_{\alpha_{ij}}$ for $i=0,\cdots,n$, and
likewise $b_i=b_{i1}+\cdots+b_{ik_i}$, where $b_{ij}\in
(R_H)_{\beta_{ij}}$ for $i=0,\cdots,k$. Put
$a_i(Y)=\sum_{j=1}^{n_i}a_{ij}Y^{\alpha_{ij}}$ for $i=0,\cdots,n$
and $b_i(Y)=\sum_{j=1}^{k_i}b_{ij}Y^{\beta_{ij}}$ for $i=0,\cdots,k$
and set $\widetilde{f}(X,Y)=\sum_{i=0}^{n}a_i(Y)X^i$ and
$\widetilde{g}(X,Y)=\sum_{i=0}^{k}b_i(Y)X^i$. Then it is easy to see
that $c(\widetilde{f})=\A_f$, $c(\widetilde{g})=\A_g$, and that
$c(\widetilde{f}\widetilde{g})=\A_{fg}$. Now using \cite[Theorem
2]{No}, there is a positive integer $m$ such that
$c(\widetilde{f})^{m+1}c(\widetilde{g})=c(\widetilde{f})^mc(\widetilde{f}\widetilde{g})$.
Hence $\A_f^{m+1}\A_g=\A_f^m\A_{fg}$. The in particular case is
clear now.
\end{proof}

\begin{lem} Set $N(\star):=\{f\in R[X]\mid f\neq0\text{ and }\A_f^{\star}=R^{\star}\}$.
Then $N(\star)$ is a saturated multiplicatively closed subset of
$R[X]$.
\end{lem}

\begin{proof} Let $f,g\in R[X]$. Then $\A_f^{m+1}\A_g=\A_f^m\A_{fg}$ for some integer $m\geq1$ by Proposition \ref{dm}, and $\A_{fg}\subseteq\A_f\A_g$. Thus $fg\in N(\star)\Leftrightarrow\A_{fg}^{\star}=R^{\star}\Leftrightarrow\A_{f}^{\star}=\A_{g}^{\star}=R^{\star}\Leftrightarrow f,g\in N(\star)$.
\end{proof}

We set $$\NA(R,\star):=R[X]_{N(\star)}$$ and call it the
\emph{Nagata ring with respect to the semistar operation $\star$}.
Obviously $\NA(R,\star)=\NA(R,\star_f)=\NA(R,\widetilde{\star})$. It
is easy to see that if $R$ has trivial grading, then
$\NA(R,\star)=\Na(R,\star)$, but they are not the same in general
(see Example \ref{E}).

\begin{prop}\label{NA} Let $R=\bigoplus_{\alpha\in\Gamma}R_{\alpha}$ be a graded integral domain, and $\star$ be a semistar operation on $R$ such that $R^{\star}\subsetneq R_H$. Then
\begin{enumerate}
\item $N(\star)=R[X]\backslash\bigcup\{Q[X]\mid Q\in h\text{-}\QMax^{\widetilde{\star}}(R)\}$.
\item $\Max(\NA(R,\star))=\{Q\NA(R,\star)\mid Q\in h\text{-}\QMax^{\widetilde{\star}}(R)\}$.
\item $\NA(R,\star)=\bigcap\{R_Q(X)\mid Q\in h\text{-}\QMax^{\widetilde{\star}}(R)\}$.
\end{enumerate}
\end{prop}

\begin{proof} (1) Let $f\in R[X]$. Then
$f\in N(\star) \Leftrightarrow
\A_f^{\widetilde{\star}}=R^{\widetilde{\star}}\Leftrightarrow
\A_f\nsubseteq Q$ for each $Q\in
h$-$\QMax^{\widetilde{\star}}(R)\Leftrightarrow f\notin Q[X]$ for
each $Q\in h$-$\QMax^{\widetilde{\star}}(R)\Leftrightarrow f\in
R[X]\backslash\bigcup\{Q[X]\mid Q\in
h\text{-}\QMax^{\widetilde{\star}}(R)\}$.

(2) Using \cite[ Proposition 4.8]{G}, it is sufficient to show that
each prime ideal $T$ of $R[X]$ contained inside $\bigcup\{Q[X]\mid
Q\in h\text{-}\QMax^{\widetilde{\star}}(R)\}$ is contained in $Q[X]$
for some $Q\in h\text{-}\QMax^{\widetilde{\star}}(R)$. Let $\A_T$ be
the ideal generated by $\{\A_f\mid f\in T\}$. It is easy to see that
$\A_T$ is a homogeneous ideal of $R$ and that if
$T\subseteq\bigcup\{Q[X]\mid Q\in
h\text{-}\QMax^{\widetilde{\star}}(R)\}$, then
$\A_T^{\widetilde{\star}}\neq R^{\widetilde{\star}}$. Indeed if
$\A_T^{\widetilde{\star}}=R^{\widetilde{\star}}$, then we can find a
polynomial $\ell\in \A_T[X]$ such that
$\A_{\ell}^{\widetilde{\star}}=R^{\widetilde{\star}}$. Now
$$
\ell\in \A_{t_1}[X]+\cdots+\A_{t_r}[X]=(\A_{t_1}+\cdots+\A_{t_r})[X]
$$
with $(t_1,\ldots,t_r)\subseteq T$. Since $\A_{t_1}+\cdots+\A_{t_r}\subseteq \A_T$ and $\A_T$ is an ideal of $R$, then $\A_{t_1}+\cdots+\A_{t_r}=\A_t$, for some $t\in T$. Therefore $\A_{\ell}\subseteq \A_t$ and thus $\A_{\ell}^{\widetilde{\star}}=\A_t^{\widetilde{\star}}=R^{\widetilde{\star}}$. This is a contradiction, since $t\in T$ and thus $\A_t^{\widetilde{\star}}\subseteq Q$, for some $Q\in h\text{-}\QMax^{\widetilde{\star}}(R)$. So that $\A_T^{\widetilde{\star}}\neq R^{\widetilde{\star}}$ and there exists $Q\in h\text{-}\QMax^{\widetilde{\star}}(R)$ such that $\A_T\subseteq Q$. This implies that $T\subseteq Q[X]$, for some $Q\in h\text{-}\QMax^{\widetilde{\star}}(R)$.

(3) is an easy consequence of (2), since
$$
(R[X]_{N(\star)})_{QR[X]_{N(\star)}}=R[X]_{Q[X]}=R_Q(X),
$$
by \cite[Corollary 5.3 and Proposition 33.1]{G}.
\end{proof}

\begin{prop}\label{N} Let $R=\bigoplus_{\alpha\in\Gamma}R_{\alpha}$ be a graded integral domain,
and $\star$ be a semistar operation on $R$ such that
$R^{\star}\subsetneq R_H$. Then for each $I\in
\overline{\mathcal{F}}(R)$, we have
\begin{enumerate}
\item $I\NA(R,\star)=\bigcap\lbrace IR_Q(X)\mid Q\in h$-$\QMax^{\widetilde{\star}}(R)\rbrace$.
\item $I\NA(R,\star)\cap K=\bigcap\{IR_Q\mid Q\in h$-$\QMax^{\widetilde{\star}}(R)\}$.
\item If $I$ is homogeneous ideal of $R$, then $I^{\widetilde{\star}}=I\NA(R,\star)\cap R_H$.
\end{enumerate}
\end{prop}

\begin{proof} (1) By Proposition \ref{NA}, we have
\begin{align*}
        I\NA(R,\star) &=\bigcap\{(I\NA(R,\star))_M\mid M\in\Max(\NA(R,\star))\}\\
        &=\bigcap\{(IR[X]_{N(\star)})_{QR[X]_{N(\star)}}\mid Q\in h\text{-}\QMax^{\widetilde{\star}}(R)\}\\
        &=\bigcap\{IR[X]_{Q[X]}\mid Q\in h\text{-}\QMax^{\widetilde{\star}}(R)\}\\
        &=\bigcap\{IR_Q(X)\mid Q\in h\text{-}\QMax^{\widetilde{\star}}(R)\}.
\end{align*}
(2) By using (1) and \cite[Proposition 33.1(4)]{G}, we have
\begin{align*}
        I\NA(R,\star)\cap K &=\bigcap\{IR_Q(X)\mid Q\in h\text{-}\QMax^{\widetilde{\star}}(R)\}\cap K\\
        &=\bigcap\{IR_Q(X)\cap K\mid Q\in h\text{-}\QMax^{\widetilde{\star}}(R)\}\\
        &=\bigcap\{IR_Q\mid Q\in h\text{-}\QMax^{\widetilde{\star}}(R)\}.
\end{align*}
(3) By \cite[Proposition 2.6]{S}, we have
\begin{align*}
        I^{\widetilde{\star}}&=\bigcap\{IR_{H\setminus Q}\mid Q\in h\text{-}\QSpec^{\widetilde{\star}}(R)\}\\
        &=\bigcap\{IR_Q\cap R_H\mid Q\in h\text{-}\QSpec^{\widetilde{\star}}(R)\}\\
        &=\bigcap\{IR_Q\mid Q\in h\text{-}\QMax^{\widetilde{\star}}(R)\}\cap R_H\\
        &=I\NA(R,\star)\cap R_H.
\end{align*}
The forth equality follows from (2).
\end{proof}

\begin{lem}\label{inv} Let $R=\bigoplus_{\alpha\in\Gamma}R_{\alpha}$ be a graded integral domain, and $\star$ be a semistar operation on $R$ such that $R^{\star}\subsetneq R_H$. Then for each nonzero finitely generated homogeneous ideal $I$ of $R$, $I$ is $\star_f$-invertible if and only if, $I\NA(R,\star)$
is invertible.
\end{lem}

\begin{proof} Let $I$ be nonzero finitely
generated homogeneous ideal of $R$, such that $I$ is
$\star_f$-invertible. Let $Q\NA(R,\star)\in\Max(\NA(R,\star))$, where $Q\in
h$-$\QMax^{\widetilde{\star}}(R)$ by Proposition \ref{NA}(2). Thus by
\cite[Theorem 2.23]{FP},
$(I\NA(R,\star))_{Q\NA(R,\star)}=IR_Q(X)$ is invertible (is
principal) in $R_Q(X)$. Hence $I\NA(R,\star)$ is invertible by
\cite[Theorem 7.3]{G}. Conversely, assume that $I$ is finitely
generated, and $I\NA(R,\star)$ is invertible. By flatness we
have
$I^{-1}\NA(R,\star)=(R:I)\NA(R,\star)=(\NA(R,\star):I\NA(R,\star))=(I\NA(R,\star))^{-1}$.
Therefore, $(II^{-1})\NA(R,\star)=(I\NA(R,\star))
(I^{-1}\NA(R,\star))\\=(I\NA(R,\star))(I\NA(R,\star))^{-1}=\NA(R,\star)$.
Hence $II^{-1}\cap N(\star)\neq\emptyset$. Let $f\in II^{-1}\cap
N(\star)$. So that
$R^{\star}=A_f^{\star}\subseteq(II^{-1})^{\star_f}\subseteq
R^{\star}$. Thus $I$ is $\star_f$-invertible.
\end{proof}

\begin{cor}\label{C(f)} Let $R=\bigoplus_{\alpha\in\Gamma}R_{\alpha}$ be a graded integral domain, $\star$ be a semistar operation on $R$ such that $R^{\star}\subsetneq R_H$ and $0\neq f\in R[X]$. Then the following conditions are equivalent:
\begin{enumerate}
\item $\A_f$ is $\star_f$-invertible.
\item $\A_f\NA(R,\star)$ is invertible.
\item $\A_f\NA(R,\star)=f\NA(R,\star)$.
\end{enumerate}
\end{cor}

\begin{proof} $(1)\Leftrightarrow(2)$ This follows from Lemma \ref{inv} because $\A_f$ is homogeneous.

$(2)\Rightarrow(3)$ Let $f=f_0+f_1X+\cdots+f_nX^n$, and $f_i=a_{i1}+\cdots+a_{ik_i}$, where each $a_{ij}$ is a homogeneous element for $i=0,\ldots,n$. By Proposition \ref{NA}, every maximal ideal of $\NA(R,\star)$ has the form $Q\NA(R,\star)$ for a homogeneous maximal quasi-$\widetilde{\star}$-ideal $Q$ of $R$; so it suffices to show that $\A_fR_Q(X)=fR_Q(X)$ \cite[Theorem 4.10]{G}.

Note that $\A_fR_Q(X)=(\A_f\NA(R,\star))_{Q\NA(R,\star)}=a_{ij}R_Q(X)$ for some $a_{ij}$ by \cite[Proposition 7.4]{G}. Thus $\frac{a_{st}}{a_{ij}}\in R_Q(X)$ for all $a_{st}$; so $\frac{f}{a_{ij}}\in R_Q(X)$. Hence $\frac{f}{a_{ij}}=\frac{g}{h}$, and $fh=ga_{ij}$ for some $g\in R[X]$ and $h\in R[X]\backslash Q[X]$. By Proposition \ref{dm}, There is an integer $m\geq1$ such that $\A_h^{m+1}\A_f=\A_h^m\A_{hf}=\A_h^m\A_{a_{ij}g}=a_{ij}\A_h^m\A_g$. Hence $a_{ij}R_Q(X)=\A_fR_Q(X)=(\A_h^{m+1}\A_f)R_Q(X)=(a_{ij}\A_h^m\A_g)R_Q(X)=(a_{ij}\A_g)R_Q(X)$, and so $\A_gR_Q(X)=R_Q(X)$, which means that $\frac{1}{g}\in R_Q(X)$. Hence $\frac{f}{a_{ij}}=\frac{g}{h}$ is a unit in $R_Q(X)$. Therefore $\A_fR_Q(X)=fR_Q(X)$.

$(3)\Rightarrow(2)$ is clear.
\end{proof}

It is known that any invertible ideal of $D(X)$ is principal \cite[Theorem 2]{A}. More generally,
if $*$ is a star operation on $D$, then any invertible ideal of $\Na(D,*)$ is principal,
\cite[Theorem 2.14]{Kang}. Now we give a generalization of these results.

\begin{thm}\label{pic} Let $R=\bigoplus_{\alpha\in\Gamma}R_{\alpha}$ be a graded integral domain, and $\star$ be a semistar operation on $R$ such that $R^{\star}\subsetneq R_H$. Then any invertible ideal of $\NA(R,\star)$ is principal.
\end{thm}

\begin{proof} Let $A\subseteq\NA(R,\star)$ be an invertible ideal. Then $A$ is finitely generated. Hence $A=(f_1,\ldots,f_n)_{N(\star)}$ for some $f_1,\ldots,f_n\in R[X]$. Let $Q$ be a homogeneous maximal quasi-$\widetilde{\star}$-ideal of $R$. Then $A_{Q[X]_{N(\star)}}=((f_1,\ldots,f_n)_{N(\star)})_{Q[X]_{N(\star)}}=(f_1,\ldots,f_n)R_Q(X)$ is invertible. Let
$$
f:=f_1+f_2X^{\partial f_1+1}+\cdots+f_nX^{\partial f_1+\cdots+\partial f_{n-1}+n-1},
$$
where $\partial f_i$ is the degree of $f_i$ for $i=1,\ldots,n$. Then by the proof of \cite[Theorem 2]{A}, we have $(f_1,\ldots,f_n)R_Q(X)=fR_Q(X)$. Hence $A=f\NA(R,\star)$ and $A$ is principal.
\end{proof}

Let $R=\bigoplus_{\alpha\in\Gamma}R_{\alpha}$ be a graded integral
domain, and $T$ be a homogeneous overring of $R$. Let $\star$ and
$\star'$ be semistar operations on $R$ and $T$, respectively. Recall from \cite{S}
that \emph{$T$ is homogeneously $(\star,\star')$-linked overring of
$R$} if
$$F^{\star}=D^{\star}\Rightarrow (FT)^{\star'}=T^{\star'}$$ for each nonzero homogeneous finitely generated ideal $F$ of
$R$. We say that \emph{$T$ is homogeneously $t$-linked overring of
$R$} if $T$ is homogeneously $(t,t)$-linked overring of $R$. Also it
can be seen that $T$ is homogeneously $(\star,\star')$-linked
overring of $R$ if and only if $T$ is homogeneously
$(\widetilde{\star},\widetilde{\star'})$-linked overring of $R$ (cf.
\cite[Theorem 3.8]{EF}).

\begin{lem}\label{link} Let $R=\bigoplus_{\alpha\in\Gamma}R_{\alpha}$ be a
graded integral domain, and let $T$ be
a homogeneous overring of $R$. Let $\star$ (resp. $\star'$) be a
semistar operation on $R$ (resp. on $T$). Then, $T$ is a
homogeneously $(\star,\star')$-linked overring of $R$ if and only if
$\NA(R,\star)\subseteq \NA(T,\star')$.
\end{lem}

\begin{proof} Let $f\in R$ such that $\A_f^{\star}=R^{\star}$.
Then by assumption $\A_{fT}^{\star'}=(\A_fT)^{\star'}=T^{\star'}$.
Hence $\NA(R,\star)\subseteq \NA(T,\star')$. Conversely let
$F$ be a nonzero homogeneous finitely generated ideal of $R$ such
that $F^{\star}=R^{\star}$. We can choose an element $f\in R$ such that $\A_f=F$. From the
fact that $\A_f^{\star}=R^{\star}$, we have that $f$ is a unit in
$\NA(R,\star)$ and so by assumption, $f$ is a unit in
$\NA(T,\star')$. This implies that
$(\A_fT)^{\star'}=\A_{fT}^{\star'}=T^{\star'}$, i.e.,
$(FT)^{\star'}=T^{\star'}$.
\end{proof}

\subsection{Kronecker function ring}

We call the Kronecker function ring with respect to $\star$ by the following:
$$
\Kr(R,\star):=\left\{\frac{f}{g}\bigg| \begin{array}{l} f,g\in R[X],
g\neq0,\text{ and there is }0\neq h\in R[X]
\\\text{ such that } \A_f\A_h\subseteq(\A_g\A_h)^{\star} \end{array} \right\}.
$$
Note that if $R$ has trivial grading, then $\Kr(R,\star)$ coincides
with the usual Kronecker function ring $\mbox{Kr}(R,\star)$,
recalled in the introduction, but they are not the same in general
(see Example \ref{E}).

\begin{lem}\label{Kr} Assume that $\star$ is an \texttt{e.a.b.} semistar
operation. Then
$$
\Kr(R,\star)=\left\{\frac{f}{g}\bigg| \begin{array}{l} f,g\in
R[X],\text{ }g\neq0,\text{ and }\A_f\subseteq \A_g^{\star} \end{array}
\right\}.
$$
and
\begin{enumerate}
\item $\Kr(R,\star)$ is an integral domain.
\item $\Kr(R,\star)$ is a B\'{e}zout domain.
\item $I^{\star}=I\Kr(R,\star)\cap R_H$ for every nonzero finitely
generated homogeneous ideal $I$ of $R$.
\end{enumerate}
\end{lem}

\begin{proof} Assume that $\star$ is an \texttt{e.a.b.} semistar
operation. In this case, for $f,g,h\in R[X]\backslash\{0\}$ we have
$$
\A_f\A_h\subseteq(\A_g\A_h)^{\star}\Leftrightarrow \A_f\subseteq \A_g^{\star}.
$$
Thus $\Kr(R,\star)=\left\{\frac{f}{g}\bigg| \begin{array}{l} f,g\in
R[X],\text{ }g\neq0,\text{ and }\A_f\subseteq \A_g^{\star} \end{array}
\right\}$.

(1) We first show that $\Kr(R,\star)$ is well defined. Assume that $f, g, s, t\in R[X]\backslash\{0\}$, such that $\frac{f}{g}=\frac{s}{t}$ and that $\A_f\subseteq \A_g^{\star}$. We show that $\A_s\subseteq \A_t^{\star}$. We have $ft=gs$. Using Proposition \ref{dm}, we have $\A_{ft}^{\star}=(\A_f\A_t)^{\star}$ and  $A_{gs}^{\star}=(A_gA_s)^{\star}$. So that
$$
(\A_g\A_s)^{\star}=\A_{gs}^{\star}=\A_{ft}^{\star}=(\A_f\A_t)^{\star}\subseteq(\A_g^{\star}\A_t)^{\star}=(\A_g\A_t)^{\star}.
$$
Now since $\star$ is \texttt{e.a.b.}, we obtain that $\A_s\subseteq \A_t^{\star}$.
Let $0\neq f_1, f_2, g_1, g_2\in R[X]$ be such that $\frac{f_1}{g_1}, \frac{f_2}{g_2}\in \Kr(R,\star)$; so $\A_{f_i}\subseteq \A_{g_i}^{\star}$ for $i=1,2$. Then $\frac{f_1}{g_1}+\frac{f_2}{g_2}=\frac{f_1g_2+f_2g_1}{g_1g_2}$ and $(\frac{f_1}{g_1})(\frac{f_2}{g_2})=\frac{f_1f_2}{g_1g_2}$. Note that, by Proposition \ref{dm} we have $\A_{f_1g_2+f_2g_1}\subseteq \A_{f_1g_2}+\A_{f_2g_1}\subseteq (\A_{f_1g_2}+\A_{f_2g_1})^{\star}=(\A_{f_1g_2}^{\star}+\A_{f_2g_1}^{\star})^{\star}=((\A_{f_1}\A_{g_2})^{\star}+(\A_{f_2}\A_{g_1})^{\star})^{\star}
\subseteq(\A_{g_1}\A_{g_2})^{\star}=\A_{g_1g_2}^{\star}$ and $\A_{f_1f_2}\subseteq \A_{f_1f_2}^{\star}=(\A_{f_1}A_{f_2})^{\star}\subseteq (\A_{g_1}\A_{g_2})^{\star}
=\A_{g_1g_2}^{\star}$. Thus $\frac{f_1g_2+f_2g_1}{g_1g_2}$ and $\frac{f_1f_2}{g_1g_2}$ are in $\Kr(R,\star)$. Hence $\Kr(R,\star)$ is an integral domain.

(2) To prove that $\Kr(R,\star)$ is a B\'{e}zout domain, assume that
$\alpha$ and $\beta$ are non-zero elements of $\Kr(R,\star)$. Write
$\alpha=\frac{f}{h}$ and $\beta=\frac{g}{h}$, where $f,g,h\in
R[X]\setminus\{0\}$. Let $n$ be a positive integer greater than the
degree of $f$. We set $\gamma=\alpha+X^n\beta$. We show that
$(\alpha,\beta)=(\gamma)$. By the choice of $n$ ,
$\A_{f+X^ng}=\A_f+\A_g$. Thus
$\frac{\alpha}{\gamma}=\frac{f}{f+X^ng}$ is in $\Kr(R,\star)$ and
similarly $\frac{\beta}{\gamma}=\frac{f}{f+X^ng}$ is in
$\Kr(R,\star)$. Therefore we showed that $(\alpha,\beta)=(\gamma)$.

(3) Assume that $I=(a_0,a_1,\ldots,a_n)$ where $a_i$ is a
homogeneous element of $R$ for $i=0,1\ldots,n$. The proof of part
(2) shows that $I\Kr(R,\star)=(a_0+a_1X+\cdots+a_nX^n)\Kr(R,\star)$.
Therefore for $d\in R_H$, one has $d\in I\Kr(R,\star)$ if and only
if $d/(a_0+a_1X+\cdots+a_nX^n)\in \Kr(R,\star)$; that is, if and
only if $(d)\subseteq
\A_{a_0+a_1X+\cdots+a_nX^n}^{\star}=(a_0,a_1,\ldots,a_n)^{\star}=I^{\star}$
\end{proof}

\begin{thm}\label{p} Let $R=\bigoplus_{\alpha\in\Gamma}R_{\alpha}$ be a graded integral domain, and $\star$ be a semistar operation on $R$ such that $R^{\star}\subsetneq R_H$ and that $\star$ is homogeneous preserving. Then
\begin{enumerate}
\item $\Kr(R,\star)=\Kr(R,\star_a)$.
\item $\Kr(R,\star)$ is a B\'{e}zout domain.
\item $I^{\star_a}=I\Kr(R,\star)\cap R_H$ for every nonzero finitely
generated homogeneous ideal $I$ of $R$.
\end{enumerate}
\end{thm}

\begin{proof} It it clear from the definition that $\Kr(R,\star)=\Kr(R,\star_f)$. Thus
we can assume, without loss of generality,
that $\star$ is a semistar operation of finite type.

Parts (2) and (3) are direct consequences of (1) using Lemma
\ref{Kr}.

For the proof of (1) assume that $\star$ is an \texttt{e.a.b.}
semistar operation of finite type. So that in this case
$\star=\star_a$ by \cite[Proposition 4.5(5)]{FL2} and hence (1) is
true. For the general case let $\star$ be a semistar operation of
finite type on $R$. By definition it is easy to see that, given two
semistar operations on $R$ with $\star_1\leq\star_2$, then
$\Kr(R,\star_1)\subseteq\Kr(R,\star_2)$. Using \cite[Proposition
4.5(3)]{FL2} we have $\star\leq\star_a$. Therefore
$\Kr(R,\star)\subseteq\Kr(R,\star_a)$. Conversely let
$\frac{f}{g}\in\Kr(R,\star_a)$. Then, $\A_f\subseteq
\A_g^{\star_a}$. Set $A:=\A_f$ and $B:=\A_g$. Then $A\subseteq
B^{\star_a}=\bigcup\{((BH)^{\star}:H)\mid H\in f(R)\}$. Suppose that
$A$ is generated by homogeneous elements $x_1,\cdots,x_n\in R$. Then
there is $H_i\in f(R)$, such that $x_iH_i\subseteq (BH_i)^{\star}$
for $i=1,\cdots,n$. Choose $0\neq r_i\in R$ such that
$F_i=r_iH_i\subseteq R$. Thus $x_iF_i\subseteq (BF_i)^{\star}$.
Therefore \cite[Lemma 3.2]{S} gives a homogeneous $T_i\in f(R)$ such
that $x_iT_i\subseteq (BT_i)^{\star}$. Set $T:=T_1T_2\cdots T_n$
which is a finitely generated homogeneous fractional ideal of $R$
such that $AT\subseteq(BT)^{\star}$. Now we can find an element
$h\in R[X]$ such that $\A_h=T$. Then
$\A_f\A_h\subseteq(\A_g\A_h)^{\star}$. This means that
$\frac{f}{g}\in\Kr(R,\star)$ to complete the proof of (1).
\end{proof}

\begin{rem}
{\em We note that the above theorem is the only result that we need
the semistar operation $\star$ is homogeneous preserving. It is not
clear for us whether, in general, this hypothesis is crucial.
However, since $\NA(R,\star)=\NA(R,\widetilde{\star})$ and
$\widetilde{\star}$ is homogeneous preserving for any semistar
operation $\star$ \cite[Proposition 2.3]{S}, we do not need to
assume that $\star$ is homoheneous preserving in this case. Also in
working with Kronecker function rings, we only need to use
$\Kr(R,\widetilde{\star})$ rather than $\Kr(R,\star)$. These are why
we do not need to assume that the semistar operations are
homogeneous preserving in most results.}
\end{rem}

\section{Graded P$\star$MDs}

In this section we make use of the Nagata ring $\NA(R,\star)$ and
the Kronecker function ring $\Kr(R,\star)$ with respect to $\star$
to give new characterizations of GP$\star$MDs.

The following theorem generalizes \cite[Corollary 28.6]{G} and \cite[Theorem 4.2]{S}.

\begin{thm}\label{cp} Let $R=\bigoplus_{\alpha\in\Gamma}R_{\alpha}$ be a
graded integral domain. Then $R$ is a graded-Pr\"{u}fer domain if
and only if $\A_f\A_g=\A_{fg}$ for all $f, g\in R[X]$.
\end{thm}

\begin{proof} $(\Rightarrow)$ Let $f, g\in R[X]$. Then by Proposition
\ref{dm}, there exists some positive integer $m$ such that
$\A_f^{m+1}\A_g=\A_f^m\A_{fg}$. Now since $R$ is a graded-Pr\"{u}fer
domain, the homogeneous fractional ideal $\A_f^m$ is invertible.
Thus $\A_f\A_g=\A_{fg}$.

$(\Leftarrow)$ Assume that $\A_f\A_g=\A_{fg}$ for all $f, g\in
R[X]$. Thus in particular $c(f)c(g)=c(fg)$ for all $f,g\in R_0[X]$.
Hence $R_0$ is a Pr\"{u}fer domain by \cite[Corollary 28.6]{G} and
in particular is integrally closed. On the other hand we have
$C(x)C(y)=C(xy)$ for all $x,y\in R$. Therefore $R$ is integrally
closed by \cite[Corollary 3.6]{AA}. Now let $a, b\in H$ be
arbitrary. So that $\A_{a+X^nb}=(a,b)$ for some integer $n\geq1$.
Since $(a+X^nb)(a-X^nb)=a^2-(X^nb)^2$, we have
$(a,b)(a,-b)=(a^2,-b^2)$. Consequently $(a,b)^2=(a^2,b^2)$. Thus by
\cite[Proposition 4.1]{S}, we see that $R$ is a graded-Pr\"{u}fer
domain.
\end{proof}

As we saw in the above proof, if
$R=\bigoplus_{\alpha\in\Gamma}R_{\alpha}$ is a graded-Pr\"{u}fer
domain, then $R_0$ is a Pr\"{u}fer domain.

The following theorem generalizes \cite[Theorem 3.7]{Ch}, \cite[Theorem 1.1]{AFZ} and \cite[Theorem 4.5]{S}.

\begin{thm}\label{hhh} Let $R=\bigoplus_{\alpha\in\Gamma}R_{\alpha}$ be a
graded integral domain, and $\star$ be
a semistar operation on $R$ such that $R^{\star}\subsetneq R_H$.
Then $R$ is a GP$\star$MD if and only if
$(\A_f\A_g)^{\widetilde{\star}}=\A_{fg}^{\widetilde{\star}}$ for all
$f, g\in R_H[X]$.
\end{thm}

\begin{proof} $(\Rightarrow)$ Let $f, g\in R_H[X]$. Choose a positive integer $m$ such that
$\A_f^{m+1}\A_g=\A_f^m\A_{fg}$ by Proposition \ref{dm}. Thus
$(\A_f^{m+1}\A_g)^{\widetilde{\star}}=(\A_f^m\A_{fg})^{\widetilde{\star}}$.
Since $R$ is a GP$\star$MD, the homogeneous fractional ideal
$\A_f^m$ is ${\widetilde{\star}}$-invertible. Thus
$(\A_f\A_g)^{\widetilde{\star}}=\A_{fg}^{\widetilde{\star}}$.

$(\Leftarrow)$ Assume that
$(\A_f\A_g)^{\widetilde{\star}}=\A_{fg}^{\widetilde{\star}}$ for all
$f, g\in R_H[X]$. Let $P\in h$-$\QMax^{\widetilde{\star}}(R)$. Then
by \cite[Proposition 2.6]{S}, we have $$\A_fR_{H\backslash
P}\A_gR_{H\backslash P}=\A_f\A_gR_{H\backslash
P}=(\A_f\A_g)^{\widetilde{\star}}R_{H\backslash
P}=\A_{fg}^{\widetilde{\star}}R_{H\backslash
P}=\A_{fg}R_{H\backslash P}.$$ The second and forth equality follow
from \cite[Proposition 2.6]{S}. Thus Theorem \ref{cp} shows that
$R_{H\backslash P}$ is a graded-Pr\"{u}fer domain. Now \cite[Theorem
4.4]{S}, implies that $R$ is a GP$\star$MD.
\end{proof}

It is a natural question that when
$(\A_f\A_g)^{\star}=\A_{fg}^{\star}$ for all $f, g\in R_H[X]$ and a
semistar operation $\star$. In Theorem \ref{hhh}, we have the answer
for $\widetilde{\star}$. In the following result we give an answer
for the case of (semi)star operations and borrow the technique from
\cite[Theorem 1.6]{AK}. Recall that the $b$-semistar operation on
$R$ is defined as $F^b:=\bigcap FV_{\alpha}$ for $V_{\alpha}$
varying in the set of all valuation overrings of $R$, and for $F\in
\overline{\mathcal{F}}(R)$.

\begin{thm}\label{in} Let $R=\bigoplus_{\alpha\in\Gamma}R_{\alpha}$ be a graded integral domain. Then the following are equivalent.
\begin{enumerate}
\item $R$ is integrally closed.
\item $(\A_f\A_g)^v=\A_{fg}^v$ for all $f,g\in R[X]\setminus\{0\}$.
\item $(\A_f\A_g)^{\star}=\A_{fg}^{\star}$  for all nonzero $f, g\in R_H[X]$ with $f$ linear and $\star$ some (semi)star operation on $R$.
\item $fR_H[X]\cap R[X]=f\A_f^{-1}R[X]$ for all nonzero $f\in R[X]$.
\end{enumerate}
\end{thm}
\begin{proof} $(1)\Rightarrow(2)$ Let $f,g\in R[X]\setminus\{0\}$. Since $b$ is an \texttt{e.a.b.} (semi)star operation on $R$, we have $(\A_f\A_g)^b=\A_{fg}^b$ by Proposition \ref{dm}. Now since $b\leq v$ we have $(\A_f\A_g)^v=((\A_f\A_g)^b)^v=(\A_{fg}^b)^v=\A_{fg}^v$.

$(2)\Rightarrow(3)$ Is clear.

$(3)\Rightarrow(1)$ Let $\alpha\in R_H$ be integral over $R$ and let
$p(X)\in R[X]$ be a monic polynomial with $p(\alpha)=0$. Hence there
is a polynomial $g\in R_H[X]$ such that $p(X)=(X-\alpha)g(X)$. Now
$\alpha\in
(\A_{(X-\alpha)}\A_g)^{\star}=(\A_{(X-\alpha)g})^{\star}=\A_p^{\star}=R^{\star}=R$.
So that $R$ is integrally closed in $R_H$. Therefore by
\cite[Proposition 5.4]{AA2} $R$ is integrally closed.

$(2)\Rightarrow(4)$ For any graded integral domain $R$ and $0\neq
f\in R[X]$ we have $f\A_f^{-1}R[X]\subseteq fR_H[X]\cap R[X]$. Let
$0\neq g\in R_H[X]$ with $fg\in fR_H[X]\cap R[X]$. Then
$\A_f\A_g\subseteq (\A_f\A_g)^v=\A_{fg}^v\subseteq R$ so
$\A_g\subseteq \A_f^{-1}$ and hence $fg\in f\A_f^{-1}R[X]$.

$(4)\Rightarrow(2)$ We have $f\A_f^{-1}R[X]=fR_H[X]\cap
R[X]\supseteq fgR_H[X]\cap R[X]=fg\A_{fg}^{-1}R[X]$. Now
$fg(\A_{fg})^{-1}R[X]\subseteq f\A_f^{-1}R[X]$ gives
$g(\A_{fg})^{-1}R[X]\subseteq\A_f^{-1}R[X]$ and hence
$\A_g(\A_{fg})^{-1}\subseteq\A_f^{-1}$. So
$\A_f\A_g(\A_{fg})^{-1}\subseteq\A_f\A_f^{-1}\subseteq R$ and hence
$(\A_{fg})^{-1}\subseteq(\A_f\A_g)^{-1}$. Thus
$(\A_f\A_g)^v\subseteq\A_{fg}^v$ and so $(\A_f\A_g)^v=\A_{fg}^v$.
\end{proof}

\begin{cor} Let $R=\bigoplus_{\alpha\in\Gamma}R_{\alpha}$ be an integrally closed
graded integral domain. Then $R$ is a graded-Pr\"{u}fer domain if
and only if every nonzero finitely generated homogeneous ideal of
$R$ is a $v$-ideal.
\end{cor}

\begin{proof} By hypothesis and Theorem \ref{in}, we
have $\A_f\A_g=(\A_f\A_g)^v=\A_{fg}^v=\A_{fg}$ for all $f, g\in
R[X]$. Thus by Theorem \ref{cp}, $R$ is a graded-Pr\"{u}fer domain.
\end{proof}

We now recall the notion of $\star$-valuation overring (a notion due
essentially to P. Jaffard \cite[page 46]{Jaf}). For a domain $D$ and
a semistar operation $\star$ on $D$, we say that a valuation
overring $V$ of $D$ is a \emph{$\star$-valuation overring of $D$}
provided $F^{\star}\subseteq FV$, for each $F\in f(D)$.

\begin{rem}\label{r} {\em (1) Let $\star$ be a semistar operation on a graded integral domain
$R=\bigoplus_{\alpha\in\Gamma}R_{\alpha}$. Recall that for each
$F\in f(R)$ we have
$$
F^{\star_a}=\bigcap\lbrace FV\mid V\text{ is a
}\star\text{-valuation overring of }R\rbrace,
$$
by \cite[Propositions 3.3 and 3.4 and Theorem 3.5]{FL1}.

(2) We have $N(\star)=N((\widetilde{\star})_a)$. Indeed, since
$\widetilde{\star}\leq(\widetilde{\star})_a$ by \cite[Proposition
4.5]{FL2}, we have $N(\star)=N(\widetilde{\star})\subseteq
N((\widetilde{\star})_a)$. Now if $f\in R[X]\backslash N(\star)$
then, $\A_f^{\widetilde{\star}}\subsetneq R^{\widetilde{\star}}$.
Thus there is a homogeneous quasi-$\widetilde{\star}$-prime ideal
$P$ of $R$ such that $\A_f\subseteq P$. Let $V$ be a valuation
domain dominating $R_P$ with maximal ideal $M$ \cite[Corollary
19.7]{G}. Therefore $V$ is a $\widetilde{\star}$-valuation overring
of $R$ by \cite[Theorem 3.9]{FL}, and $\A_fV\subseteq M$; so
$\A_f^{(\widetilde{\star})_a}\subsetneq R^{(\widetilde{\star})_a}$
and $f\notin N((\widetilde{\star})_a)$. Thus we obtain that
$N(\star)=N((\widetilde{\star})_a)$.}
\end{rem}

In the following theorem we generalize a characterization of P$v$MDs
proved by Arnold and Brewer \cite[Theorem 3]{AB}. It also
generalizes \cite[Theorem 3.7]{Ch}, \cite[Theorems 3.4 and 3.5]{AC},
\cite[Theorem 3.1]{FJS}, and \cite[Theorem 4.7]{S}.

\begin{thm}\label{NKP1} Let $R=\bigoplus_{\alpha\in\Gamma}R_{\alpha}$ be a
graded integral domain, and $\star$ be
a semistar operation on $R$ such that $R^{\star}\subsetneq R_H$.
Then, the following statements are equivalent:
\begin{enumerate}
\item $R$ is a GP$\star$MD.
\item Every ideal of $\NA(R,\star)$ is extended from a homogeneous ideal of $R$.
\item Every principal ideal of $\NA(R,\star)$ is extended from a homogeneous ideal of $R$.
\item $\NA(R,\star)$ is a Pr\"{u}fer domain.
\item $\NA(R,\star)$ is a B\'{e}zout domain.
\item $\NA(R,\star)=\Kr(R,\widetilde{\star})$.
\item $\Kr(R,\widetilde{\star})$ is a quotient ring of $R[X]$.
\item $\Kr(R,\widetilde{\star})$ is a flat $R[X]$-module.
\item $I^{\widetilde{\star}}=I^{(\widetilde{\star})_a}$ for each nonzero homogeneous finitely
generated ideal of $R$.
\end{enumerate}
In particular if $R$ is a GP$\star$MD, then
$R^{\widetilde{\star}}$ is integrally closed.
\end{thm}

\begin{proof} By Theorem
\ref{p}, we have $\Kr(R,\widetilde{\star})$ is well-defined since $\widetilde{\star}$ is homogeneous preserving \cite[Proposition 2.3]{S}.

$(1)\Rightarrow(2)$ Let $0\neq f\in R[X]$. Then $\A_f$ is
$\widetilde{\star}$-invertible, because $R$ is a GP$\star$MD,
and thus $f\NA(R,\star)=\A_f\NA(R,\star)$ by Corollary
\ref{C(f)}. Hence if $A$ is an ideal of $\NA(R,\star)$, then
$A=I\NA(R,\star)$ for some ideal $I$ of $R[X]$, and thus
$A=(\sum_{f\in I}\A_f)\NA(R,\star)$.

$(2)\Rightarrow(3)$ Clear.

$(3)\Rightarrow(1)$ Let $a,b\in H$ and $h:=a+bX$; so $\A_h=(a,b)$.
By (3) $h\NA(R,\star)=I\NA(R,\star)$ for some homogeneous ideal $I$
of $R$. Note that $(a,b)\subseteq I\NA(R,\star)\cap
R_H=I^{\widetilde{\star}}$ by Proposition \ref{N}. Thus
$$
h\NA(R,\star)=I\NA(R,\star)=I^{\widetilde{\star}}\NA(R,\star)\supseteq\A_h\NA(R,\star)\supseteq h\NA(R,\star).
$$
Hence $h\NA(R,\star)=\A_h\NA(R,\star)$. Thus $\A_h=(a,b)$ is $\star_f$-invertible by Corollary \ref{C(f)}.

$(1)\Rightarrow(4)$ Let $A$ be a nonzero finitely generated ideal of
$\NA(R,\star)$. Then, $A=I\NA(R,\star)$ for some nonzero finitely
generated homogeneous ideal $I$ of $R$ by $(1)\Leftrightarrow(2)$.
Since $R$ is a GP$\star$MD, $I$ is $\widetilde{\star}$-invertible,
and thus $A=I\NA(R,\star)$ is invertible by Lemma \ref{inv}.

$(4)\Rightarrow(5)$ Follows from Theorem \ref{pic}.

$(5)\Rightarrow(6)$ Clearly
$\NA(R,\star)\subseteq\Kr(R,\widetilde{\star})$. Since
$\NA(R,\star)$ is a B\'{e}zout domain, then
$\Kr(R,\widetilde{\star})$ is a quotient ring of $\NA(R,\star)$,
by \cite[Proposition 27.3]{G}. If $Q\in
h$-$\QMax^{\widetilde{\star}}(R)$, then
$Q\Kr(R,\widetilde{\star})\subsetneq\Kr(R,\widetilde{\star})$.
Otherwise $Q\Kr(R,\widetilde{\star})=\Kr(R,\widetilde{\star})$, and
hence there is an element $f\in Q$, such that
$f\Kr(R,\widetilde{\star})=\Kr(R,\widetilde{\star})$. Thus
$\frac{1}{f}\in\Kr(R,\widetilde{\star})$. Therefore $R=C(1)\subseteq
\A_f^{(\widetilde{\star})_a}\subseteq R^{(\widetilde{\star})_a}$, so
that $\A_f^{(\widetilde{\star})_a}=R^{(\widetilde{\star})_a}$. Hence
$f\in N((\widetilde{\star})_a)=N(\star)$ by Remark
\ref{r}(2). This means that
$Q^{\widetilde{\star}}=R^{\widetilde{\star}}$, a contradiction. Thus
$Q\Kr(R,\widetilde{\star})\subsetneq\Kr(R,\widetilde{\star})$, and
so there is a maximal ideal $M$ of $\Kr(R,\widetilde{\star})$ such
that $Q\Kr(R,\star)\subseteq M$. Hence $M\cap
\NA(R,\star)=Q\NA(R,\star)$, by Lemma \ref{NA}.
Consequently $R_Q(X)\subseteq \Kr(R,\widetilde{\star})_M$, and since
$R_Q(X)$ is a valuation domain, we have
$R_Q(X)=\Kr(R,\widetilde{\star})_M$. Therefore
$\NA(R,\star)=\bigcap_{Q\in
h\text{-}\QMax^{\widetilde{\star}}(R)}R_Q(X)\supseteq\bigcap_{M\in\Max(\Kr(R,\widetilde{\star}))}\Kr(R,\widetilde{\star})_M$.
Hence $\NA(R,\star)=\Kr(R,\widetilde{\star})$.

$(6)\Rightarrow(7)$ and $(7)\Rightarrow(8)$ are clear.

$(8)\Rightarrow(6)$ Recall that an overring $T$ of an integral
domain $S$ is a flat $S$-module if and only if $T_M=S_{M\cap S}$ for
all $M\in\Max(T)$ by \cite[Theorem 2]{R}.

Let $A$ be an ideal of $R[X]$ such that
$A\Kr(R,\widetilde{\star})=\Kr(R,\widetilde{\star})$. Then there
exists an element $f\in A$ such that
$f\Kr(R,\widetilde{\star})=\Kr(R,\widetilde{\star})$ using Theorem
\ref{p}; so
$\frac{1}{f}\in\Kr(R,\widetilde{\star})=\Kr(R,(\widetilde{\star})_a)$.
Thus $R=C(1)\subseteq \A_f^{(\widetilde{\star})_a}\subseteq
R^{(\widetilde{\star})_a}$, and so
$\A_f^{(\widetilde{\star})_a}=R^{(\widetilde{\star})_a}$. Hence
$\A_f^{\widetilde{\star}}=R^{\widetilde{\star}}$. Therefore $f\in
A\cap N(\star)\neq\emptyset$. So that, if $P_0$ is a homogeneous
quasi-$\widetilde{\star}$-maximal ideal of $R$, then
$P_0\Kr(R,\widetilde{\star})\subsetneq\Kr(R,\widetilde{\star})$, and
since $P_0\NA(R,\star)$ is a maximal ideal of $\NA(R,\star)$, there
is a maximal ideal $M_0$ of $\Kr(R,\widetilde{\star})$ such that
$M_0\cap R=(M_0\cap \NA(R,\star))\cap R=P_0\NA(R,\star)\cap R=P_0$.
Thus by (8),
$\Kr(R,\widetilde{\star})_{M_0}=R_{P_0}(X)=(\NA(R,\star))_{P_0\NA(R,\star)}$.

Let $M_1$ be a maximal ideal of $\Kr(R,\widetilde{\star})$, and let
$P_1$ be a homogeneous quasi-$\widetilde{\star}$-maximal ideal of
$R$ such that $M_1\cap \NA(R,\star)\subseteq
P_1\NA(R,\star)$. By the above paragraph, there is a maximal
ideal $M_2$ of $\Kr(R,\widetilde{\star})$ such that
$\Kr(R,\widetilde{\star})_{M_2}=(\NA(R,\star))_{P_1\NA(R,\star)}$.
Note that
$\Kr(R,\widetilde{\star})_{M_2}\subseteq\Kr(R,\widetilde{\star})_{M_1}$
, $M_1$ and $M_2$ are maximal ideals, and $\Kr(R,\widetilde{\star})$
is a Pr\"{u}fer domain; hence $M_1=M_2$ (cf. \cite[Theorem
17.6(c)]{G}) and
$\Kr(R,\widetilde{\star})_{M_1}=(\NA(R,\star))_{P_1\NA(R,\star)}$.
Thus
$$
\Kr(R,\widetilde{\star})=\bigcap_{M\in\Max(\Kr(R,\widetilde{\star}))}\Kr(R,\widetilde{\star})_M=\bigcap_{P\in
h\text{-}\QMax^{\widetilde{\star}}(R)}(\NA(R,\star))_{P\NA(R,\star)}=\NA(R,\star).
$$

$(6)\Rightarrow(9)$ Assume that
$\NA(R,\star)=\Kr(R,\widetilde{\star})$. Let $I$ be a nonzero
homogeneous finitely generated ideal of $R$. Then by Proposition
\ref{N} and Theorem \ref{p}, we have
$I^{\widetilde{\star}}=I\NA(R,\star)\cap
R_H=I\Kr(R,\widetilde{\star})\cap R_H=I^{(\widetilde{\star})_a}$.

$(9)\Rightarrow(1)$ Let $a$ and $b$ be two nonzero homogeneous
elements of $R$. Then
$((a,b)^3)^{\widetilde{\star}_a}=((a,b)(a^2,b^2))^{\widetilde{\star}_a}$
which implies that
$((a,b)^2)^{\widetilde{\star}_a}=(a^2,b^2)^{\widetilde{\star}_a}$.
Hence $((a,b)^2)^{\widetilde{\star}}=(a^2,b^2)^{\widetilde{\star}}$
and so $(a,b)^2R_{H\backslash P}=(a^2,b^2)R_{H\backslash P}$ for
each homogeneous quasi-$\widetilde{\star}$-maximal ideal $P$ of $R$.
On the other hand $R^{\widetilde{\star}}=R^{\widetilde{\star}_a}$ by
(9). Hence $R^{\widetilde{\star}}$ is integrally closed. Thus
$R^{\widetilde{\star}}R_{H\backslash P}=R_{H\backslash P}$ is
integrally closed. Therefore by \cite[Proposition 4.1]{S},
$R_{H\backslash P}$ is a graded-Pr\"{u}fer domain for each
homogeneous quasi-$\star_f$-maximal ideal of $R$. Thus $R$ is a
GP$\star$MD by \cite[Theorem 4.4]{S}.
\end{proof}

The following theorem is a graded version of a characterization of
Pr\"{u}fer domains proved by Davis \cite[Theorem 1]{D}. It also
generalizes \cite[Theorem 2.10]{DHLZ}, in the $t$-operation,
\cite[Theorem 5.3]{EF}, in the case of semistar operations, and \cite[Theorem 4.8]{S}, in the graded case.

\begin{thm}\label{NKP2} Let $R=\bigoplus_{\alpha\in\Gamma}R_{\alpha}$ be a
graded integral domain, and $\star$ be
a semistar operation on $R$ such that $R^{\star}\subsetneq R_H$.
Then, the following statements are equivalent:
\begin{enumerate}
\item $R$ is a GP$\star$MD.
\item Each homogeneously $(\star,t)$-linked overring of $R$ is a
P$v$MD.
\item Each homogeneously $(\star,d)$-linked overring of $R$ is
a graded-Pr\"{u}fer domain.
\item Each homogeneously $(\star,t)$-linked overring of $R$, is integrally closed.
\item Each homogeneously $(\star,d)$-linked overring of $R$, is integrally closed.
\end{enumerate}
\end{thm}

\begin{proof} $(1)\Rightarrow(2)$ Let $T$ be a homogeneously $(\star,t)$-linked overring of
$R$. Thus by Lemma \ref{link}, we have $\NA(R,\star)\subseteq
\NA(T,v)$. Since $R$ is a GP$\star$MD, by Theorem
\ref{NKP1}, we have $\NA(R,\star)$ is a Pr\"{u}fer domain. Thus
by \cite[Theorem 26.1]{G}, we have $\NA(T,v)$ is a Pr\"{u}fer
domain. Hence, again by Theorem \ref{NKP1}, we have $T$ is a GP$v$MD. Therefore using \cite[Theorem 6.4]{AA}, $T$ is a P$v$MD.

$(2)\Rightarrow(4)\Rightarrow(5)$ and $(3)\Rightarrow(5)$ are clear.

$(5)\Rightarrow(1)$ Let $P\in h$-$\QMax^{\widetilde{\star}}(R)$. For
a nonzero homogeneous $u\in R_H$, let $T=R[u^2,u^3]_{H\backslash
P}$. Then $R_{H\backslash P}$ and $T$ are homogeneously
$(\star,d)$-linked overring of $R$ by \cite[Example 2.14]{S}. So
that $R_{H\backslash P}$ and $T$ are integrally closed. Hence $u\in
T$, and since $T=R_{H\backslash P}[u^2,u^3]$, there exists a
polynomial $\gamma\in R_{H\backslash P}[X]$ such that $\gamma(u)=0$
and one of the coefficients of $\gamma$ is a unit in $R_{H\backslash
P}$. So $u$ or $u^{-1}$ is in $R_{H\backslash P}$ by \cite[Theorem
67]{K}. Therefore by \cite[Lemma 4.3]{S}, $R_{H\backslash P}$ is a
graded-Pr\"{u}fer domain. Thus $R$ is a GP$\star$MD by \cite[Theorem
4.4]{S}.

$(1)\Rightarrow(3)$ Is the same argument as in part
$(1)\Rightarrow(2)$.
\end{proof}

\begin{cor}\label{ppp} Let $R=\bigoplus_{\alpha\in\Gamma}R_{\alpha}$ be a
graded integral domain. Then, the following statements are equivalent:
\begin{enumerate}
\item $R$ is a graded-Pr\"{u}fer domain.
\item Every ideal of $\NA(R,d)$ is extended from a homogeneous ideal of $R$.
\item Every principal ideal of $\NA(R,d)$ is extended from a homogeneous ideal of $R$.
\item $\NA(R,d)$ is a Pr\"{u}fer domain.
\item $\NA(R,d)$ is a B\'{e}zout domain.
\item $\NA(R,d)=\Kr(R,d)$.
\item $\Kr(R,d)$ is a quotient ring of $R[X]$.
\item $\Kr(R,d)$ is a flat $R[X]$-module.
\item Each homogeneous overring of $R$, is integrally closed.
\item Each homogeneous overring of $R$ is a graded-Pr\"{u}fer domain.
\item $I^b=I$ for each nonzero homogeneous finitely
generated ideal of $R$.
\end{enumerate}
\end{cor}

\begin{proof} Let $\star=d$ the identity operation in Theorems \ref{NKP1} and \ref{NKP2}, and note that $d_a=b$.
\end{proof}

\begin{exam}\label{E}
{\em Let $R$ be a graded-Pr\"{u}fer domain which is not Pr\"{u}fer
(e.g. $R=D[X,X^{-1}]$ for a Pr\"{u}fer domain $D$ which is not a
field  and an indeterminate $X$ over $D$ \cite[Example 3.6]{AC}).
Then $\NA(R,d)=\Kr(R,d)$ by Corollary \ref{ppp}, but $\Na(R,d)$ is
not Pr\"{u}fer \cite[Theorem 33.4]{G}. Thus $\NA(R,d)\neq\Na(R,d)$.
Also we note that $\Kr(R,d)\neq\mbox{Kr}(R,d)$. Since otherwise
$\NA(R,d)=\mbox{Kr}(R,d)$ and so $\mbox{Kr}(R,d)$ is a quotient ring
of $R[X]$, and hence $R$ is a Pr\"{u}fer domain (\cite[Theorem
33.4]{G}) which is absurd. Finally there are $f,g\in R[X]$ such that
$c(fg)\neq c(f)c(g)$ by \cite[Theorem 28.6]{G} while
$\A_{fg}=\A_f\A_g$ by Theorem \ref{cp}.}
\end{exam}

The next corollary gives new characterizations of P$v$MDs for graded
integral domains, which is the special cases of Theorems \ref{hhh}, \ref{NKP1}, and \ref{NKP2}, for $\star=v$.

\begin{cor}\label{v} Let $R=\bigoplus_{\alpha\in\Gamma}R_{\alpha}$ be a
graded integral domain. Then, the
following statements are equivalent:
\begin{enumerate}
\item $R$ is a (graded) P$v$MD.
\item Every ideal of
$\NA(R,v)$ is extended from a homogeneous ideal of $R$.
\item $\NA(R,v)$ is a Pr\"{u}fer domain.
\item $\NA(R,v)$ is a B\'{e}zout domain.
\item $\NA(R,v)=\Kr(R,w)$.
\item $\Kr(R,w)$ is a quotient ring of $R[X]$.
\item $\Kr(R,w)$ is a flat $R[X]$-module.
\item Each homogeneously $t$-linked overring of $R$ is a P$v$MD.
\item Each homogeneously $t$-linked overring of $R$, is integrally closed.
\item $(\A_f\A_g)^w=\A_{fg}^w$ for all $f, g\in R_H[X]$.
\item $I^w=I^{w_a}$ for each nonzero homogeneous finitely
generated ideal of $R$.
\end{enumerate}
\end{cor}

In concluding we determine when $\NA(R,v)$ is a PID. Recall from
\cite[Definition 5.1]{AA} that a graded integral domain is called a
\emph{graded Krull domain} if it is completely integrally closed
(with respect to homogeneous elements) and satisfies the ascending
chain condition on homogeneous $v$-ideals.

\begin{thm}\label{krull} Let $R=\bigoplus_{\alpha\in\Gamma}R_{\alpha}$ be a
graded integral domain. Then, the following statements are equivalent:
\begin{enumerate}
\item $R$ is a graded Krull domain.
\item $\NA(R,v)$ is a Dedekind domain.
\item $\NA(R,v)$ is a PID.
\end{enumerate}
\end{thm}

\begin{proof} $(2)\Leftrightarrow(3)$ Follows from Theorem \ref{pic}.

$(1)\Rightarrow(2)$ If $R$ is a graded Krull domain, then every nonzero homogeneous ideal of $R$ is $w$-invertible, by \cite[Theorem 2.4]{AC1}. So in particular $R$ is a (graded) P$v$MD, and every nonzero homogeneous ideal of $R$ is $w$-finite. Thus by Corollary \ref{v}, $\NA(R,v)$ is a Pr\"{u}fer domain and every ideal of $\NA(R,v)$ is extended from a homogeneous ideal of $R$. Assume that $A$ is an ideal of $\NA(R,v)$. Hence $A=I\NA(R,v)$, for some homogeneous ideal $I$ of $R$. So that there is a finitely generated homogeneous ideal $J\subseteq I$ such that $I^w=J^w$. Therefore $A=I\NA(R,v)=I^w\NA(R,v)=J^w\NA(R,v)=J\NA(R,v)$ is finitely generated. Hence $A$ is finitely generated and so $\NA(R,v)$ is Noetherian. Thus $\NA(R,v)$ is a Dedekind domain.

$(2)\Rightarrow(1)$ If $\NA(R,v)$ is a Dedekind domain, then $R$ is a graded P$v$MD by Corollary \ref{v}. Let $I$ be a homogeneous ideal of $R$. Thus $I\NA(R,v)$ is finitely generated. So that there is a finitely generated homogeneous ideal $J\subseteq I$ such that $I\NA(R,v)=J\NA(R,v)$ and hence $I^w=J^w$. Thus $J$ is $w$-invertible. Note that $JI^{-1}\subseteq(JI^{-1})^w\subseteq(JJ^{-1})^w=R$. Therefore we obtain that $J^{-1}=I^{-1}$. Thus $R=(JJ^{-1})^w=(JI^{-1})^w\subseteq(II^{-1})^w\subseteq R$. This means that $I$ is $w$-invertible. Now \cite[Theorem 2.4]{AC1}, implies that $R$ is a graded Krull domain.
\end{proof}

\begin{cor} Let $R=\bigoplus_{\alpha\in\Gamma}R_{\alpha}$ be a
graded integral domain. Then, the following statements are equivalent:
\begin{enumerate}
\item $R$ is a Krull domain.
\item $R_H$ is a UFD and $\NA(R,v)$ is a PID.
\item $R_H$ and $\NA(R,v)$ are both Krull domains.
\end{enumerate}
\end{cor}
\begin{proof} $(1)\Rightarrow(2)$ A Krull domain is a graded Krull domain, and thus $\NA(R,v)$ is a PID by Theorem \ref{krull}. Also, $R_H$ is a UFD since $R_H$ is a GCD-domain by \cite[Proposition 2]{AA2} and a Krull domain.

$(2)\Rightarrow(3)$ Clear.

$(3)\Rightarrow(1)$ This follows from the fact that $R=R^w=\NA(R,v)\cap R_H$ by Proposition \ref{N}.
\end{proof}

\begin{center} {\bf ACKNOWLEDGMENT}
\end{center}

The author wish to thank the referee for an insightful report.

\end{document}